\documentclass{birkmult}

\newcommand{\Sym}{\ensuremath \mathrm{Sym}}
\newcommand{\ord}{\ensuremath \mathop \mathrm{ord} \nolimits}

 \newtheorem{thm}{Theorem}[section]

 \theoremstyle{definition}
 \newtheorem{defn}[thm]{Definition}
 \theoremstyle{remark}
 \newtheorem{rem}[thm]{Remark}
 \newtheorem{ex}[thm]{Example}
 \numberwithin{equation}{section}

\begin{document}
%
%
%
%
%
%
%
%
%
\title[Parametric Factorizations of LPDOs]
 {Parametric Factorizations of \\
  Second-, Third- and Fourth-Order \\
  Linear Partial Differential Operators \\
  with a Completely Factorable Symbol
  on the Plane }
\author[Ekaterina Shemyakova]{Ekaterina Shemyakova}

\address{%
Research Institute for Symbolic Computations (RISC)\\
Johannes Kepler University\\
Altenbergerstr. 69\\
A-4040 Linz\\
Austria
}

\email{kath@risc.uni-linz.ac.at}

\thanks{This work was supported by Austrian Science Foundation (FWF) under
the project SFB F013/F1304.}


\date{November 22, 2006}

\begin{abstract}
Parametric factorizations of linear partial operators on the plane
are considered for operators of orders two, three and four. The
operators are assumed to have a completely factorable symbol. It is
proved that ``irreducible'' parametric factorizations may exist only
for a few certain types of factorizations. Examples are given of the
parametric families for each of the possible types. For the
operators of orders two and three, it is shown that any
factorization family is parameterized by a single univariate
function
 (which can be a constant function).
\end{abstract}

\maketitle
\section{Introduction}

The factorization of the Linear Partial Differential Operator
(LPDO)
\[
L = \sum_{i_1 +  \dots + i_n \leq d} a_{i_1 \dots i_n} D_1^{i_1}
\dots D_n^{i_n} \ ,
\]
where the coefficients belong to some differential ring, is an
important technique, used by modern algorithms for the integration of the
corresponding Linear Partial Differential Equation (LPDE)
$L(f)=0$.
Many of the algorithms have appeared as advanced modifications and
generalizations of the well-known \emph{Laplace-Euler transformation
method} (\emph{the Laplace cascade method}) which serves as a method
for computing a general solution for the bivariate second-order
linear hyperbolic equations
\begin{equation} \label{hyper_nf}
 u_{xy}+a u_x + b u_y + c=0\ , \quad a=a(x,y)\,, \ b=b(x,y)\,, \ c=c(x,y)\,.
\end{equation}
The basis of the algorithm is the fact that whenever the
 operator in (\ref{hyper_nf}) factors, for instance,
\begin{equation} \label{Laplace_op}
L=D_x \circ D_y + a D_x + b D_y + c = (D_x + b) \circ (D_y + a)\ ,
\end{equation}
the problem of determining all the integrals of the equation
\eqref{hyper_nf} reduces to the problem of integrating the
two first-order equations
\begin{eqnarray*}
         (D_x+b)(u_1)&=&0\ ,\\
         (D_y+a)(u)&=&u_1\ .
\end{eqnarray*}
Accordingly, one obtains the general solution of the original equation
\eqref{hyper_nf} as
\[
z=\Big(p(x)+ \int q(y) e^{\int a dy -b dx} dy \Big)e^{-\int a dy}
\]
with two arbitrary functions $p(x)$ and $q(y)$.
If it is the case that the operator
\eqref{Laplace_op} is not factorable, one may construct (by
\emph{Laplace transformations}) the sequence of operators
\[
\dots \leftrightarrow L_{-2} \leftrightarrow L_{-1} \leftrightarrow
L \leftrightarrow L_1 \leftrightarrow L_2 \leftrightarrow \dots,
\]
such that the kernel of any descendant can be found from the kernel
of the initial operator $L$. Thus, whenever one of those operators
becomes factorable, the general solution of the initial problem can
be found.

Since the 18th century, many generalizations of the method have
been given (for example \cite{ anderson_juras, anderson97_2,athorne1995,
athorne2002, Darboux2, Dini1901, Dini1902,Juras95,Li_S_Ts,Roux1899,
sokolov, tsarev06}), and the problem of constructing
factorization algorithms for different kinds of differential
operators has become important. Over the last decade, a number of new
modifications of the classical algorithms for the factorization of
LPDOs (for example, \cite{GS, GS05, spain, dubna, tsarev98, tsarev05})
have been given. So far, most of the activity has addressed the hyperbolic
case, and there is as yet a lack of knowledge concerning the non-hyperbolic
case.

There is a distinction in kind between the two cases. A
factorization of a hyperbolic LPDO on the plane is determined
uniquely by a factorization of the operator's symbol (principal
symbol) (see Theorem \ref{GS} and \cite{GS}). Thus, the operator
\eqref{Laplace_op} may have at most one factorization of each of the
forms $(D_x+ \dots) \circ (D_y + \dots)$ and $(D_y+ \dots) \circ
(D_x + \dots)$. On the other hand for the non-hyperbolic operator
$D_{xx}$ there is the stereotypical example
\[
D_{xx}=D_x \circ D_x = \big(D_x+ \frac{1}{x+c} \big) \circ \big(D_x
- \frac{1}{x+c} \big)\ ,
\]
where $c$ is an arbitrary parameter.
A more significant example is provided by the
Landau operator $L=D_{xx}(D_x + x D_y) + 2D_{xx}+ 2(x+1)D_{xy}
+D_{x}+ (x+1)D_{y}$, which factors as
\begin{equation} \label{Landau_ex}
L= \Big(D_x + 1 + \frac{1}{x+c(y)} \Big) \circ \Big(D_x + 1 -
\frac{1}{x+c(y)} \Big) \circ \Big(D_x + x D_y \Big)\ ,
\end{equation}
where the function $c(y)$ is arbitrary.
This shows that some LPDOs
may have essentially different factorizations, and, further, that the factors
may contain arbitrary parameters or even functions.
Thus we may have \emph{families of factorizations}.

An LPDO is hyperbolic if its symbol is completely factorable (all
factors are of first order) and each factor has multiplicity one.
In the present paper we consider the case of LPDOs
that have completely factorable symbols, without any additional
requirement. We prove that ``irreducible'' (see Definition
\ref{def_irr})  families of factorizations can exist only for a few
certain types of factorizations.
For these cases explicit examples are given.
For operators of orders two and three, it is
shown that a family may be parameterized by at most one function in
one variable. Our investigations cover the case of ordinary
differential operators as well. Some related remarks about
parametric factorizations for ordinary differential operators may be
found in \cite{Tsarev_enum_ODE}.

The paper is organized as follows. In section~\ref{sec:Prelim}, we give some
general definitions and remarks about factorizations and families of
factorizations of LPDOs. Then in section $3$, we describe our basic
tool --- local linearization. In sections $4$, $5$, and $6$
we formulate and prove results about LPDOs of orders two, three, and
four respectively. The last section contains conclusions, and
some ideas for future work.
\section{Preliminaries}
\label{sec:Prelim}
 We consider a field $K$ with a set $\Delta=\{ \partial_1, \dots, \partial_n\}$
of commuting derivations acting on it. We work with the ring of
linear differential operators $K[D]=K[D_1, \dots, D_n]$, where $D_1,
\dots, D_n$ correspond to the derivations $\partial_1, \dots,
\partial_n$, respectively. Any operator $L \in K[D]$ is of the form
\begin{equation} \label{general form}
L = \sum_{|J| \leq d} a_J D^J,
\end{equation}
 where $a_J \in K, \ J \in \mathbf{N}^n$
and $|J|$ is the sum of the components of $J$. The \emph{symbol} of
$L$ is the homogeneous polynomial $ \Sym_L = \sum_{|J| = d} a_J
X^J$. The operator $L$ is \emph{hyperbolic} if the polynomial $
\Sym_L$ has exactly $d$ different factors.
\begin{defn}
Let $L, F_1, \dots, F_k \in K[D]$.
A factorization $L=F_1 \circ \dots \circ F_k$ is said to be of
the \emph{factorization type} $(S_1) \dots (S_k)$, where $S_i=\Sym_{F_i}$ for all $i$.
\end{defn}
\begin{defn} Let $L \in K[D]$. We say that
\begin{equation} \label{family}
L=F_1(T) \circ \dots \circ F_k(T)
\end{equation}
 is a \emph{family of
factorizations} of $L$ \emph{parameterized by the parameter $T$} if,
for any value $T=T_0$, we have that
 $F_1(T_0), \dots, F_k(T_0)$ are in $K[D]$ and
 $L=F_1(T_0) \circ \dots \circ F_k(T_0)$
holds. Here $T$ is an element from the space of parameters
$\mathbb{T}$. Usually $\mathbb{T}$ is the Cartesian product of some
(functions') fields, in which the number of variables is less than
that in $K$.
\end{defn}
We often consider families without mentioning or designating the
corresponding operator; we define the symbol and the order of the
family to be equal to symbol and order of the operator.
\begin{defn} \label{def_irr}
We say that a family of factorizations \eqref{family} is
\emph{reducible}, if there is $i, \ 1 \leq i \leq k$, such that the
product
\[
F_1(T) \circ \dots \circ F_i(T)
\]
does not depend on the parameter $T$ (in this case the product
$F_{i+1}(T) \circ \dots \circ F_k(T)$ does not depend on the
parameters as well). Otherwise the family is \emph{reducible}.
\end{defn}
Thus, the family \eqref{Landau_ex} is reducible. However, the
product of the first two factors does not depend on the parameter,
while the factors themselves do. So we have an example of a
second-order irreducible family of factorizations.
\begin{rem} \label{number_of_factors}
Note that any irreducible family of the type $(S_1)(S_2)(S_3)$
serves as an irreducible family of the types $(S_1 S_2)(S_3)$ and
$(S_1)(S_2 S_3)$ as well. Indeed, the irreducibility of the family
of the type $(S_1)(S_2)(S_3)$ means that the product of the first
and the second factors, as well as that of the second one and the
third one, depends on the parameter.

Analogous property enjoys the families of arbitrary orders.
\end{rem}
\begin{thm} \cite{GS} \label{GS}
 Let $L \in K[D]$, and $\Sym_L=S_1 \dots S_k$, and let the $S_i$ be pairwise coprime.
There is at most one factorization of $L$ of the type
$(S_1) \dots (S_k)$.
\end{thm}
The theorem implies that, for instance, there are no irreducible
families of the types $(X)(Y^3)$ or $(X^2)(Y^2)$.
\begin{rem}
The properties of factorizations, such as the existence of the
factorizations,
or the number of parameters, or again the number of variables in
parametric functions, are invariant under a change of variables and
the gauge transformations $L \mapsto g^{-1} L g$, $g \in K$, of the
initial operator.
\end{rem}
\begin{defn} We say that a partial differential operator $L \in K[D]$
is \emph{almost ordinary} if it is an ordinary differential operator
in some system of coordinates (transformation's functions belong to
$K$).
\end{defn}
\section{The Linearized Problem}

The basic tool in our study of families of factorizations will be
their \emph{linearization}. Let an operator $L \in K[D]$ have a
family of factorizations
\[
L=M_1(T) \circ M_2(T)\ ,
\]
parameterized by some parameters $T=(t_1, \dots, t_k)$, with
$M_1(T), M_2(T) \in K[D]$. By means of a multiplication by a
function from $K$, one can make the symbols of $M_1(T)$ and $M_2(T)$
independent of the parameters. Take some point $T_0$ as an initial
point, make the substitution $T \rightarrow T_0 + \varepsilon R$,
and equates the coefficient at the power $\varepsilon$. This implies
\begin{equation} \label{lian}
F_1 \circ L_2 + L_1 \circ F_2 =0 \ .
\end{equation}
where we have denoted the initial factorization factors by $L_i=M_i(T_0)$,
and $F_i=F_i(R)$ for $i=1,2$.
Analogously, for factorizations into three factors we get
\begin{equation} \label{lian3}
F_1 \circ L_2 \circ L_3 + L_1 \circ F_2 \circ L_3 + L_1 \circ L_2 \circ F_3 =0 \ .
\end{equation}

In the paper we apply the linearization to obtain some important
information about families of factorizations.
\section{Second-Order Operators}
\begin{thm} \label{order2}
A second-order operator in $K[D_x,D_y]$ has a family of
factorizations (in some extension of the field $K$) if and only if
it is almost ordinary. Any such family is unique for a given
operator. Further, in appropriate variables it has the form
\[
\left(D_x+ a + \frac{Q}{W+f_1(y)} \right) \circ \left(D_x + b -
\frac{Q}{W+f_1(y)} \right)\ ,
\]
where $Q=e^{\int (b-a) dx}$, $W=\int Q dx$, $a,b \in K$, and $f_1(y)
\in K$ is a parameter.
\end{thm}
\begin{proof}
Consider a second-order operator $L \in K[D_x,D_y]$. By a change of
variables we can make the symbol of $L$ equal to either $X^2$ or $XY$.
In the latter case, $L$ has no family of factorizations because
of Theorem \ref{GS}.

Consider the case $\Sym_L=X^2$. Then operator $L$ has a factorization
only if it is ordinary.
Suppose we know one factorization: $L=L_1 \circ
L_2= (D_x+ a)\circ (D_x + b)$, where $a,b \in K$, and we are
interested in deciding whether there exists a family.
Consider the linearized
problem, that is the equation (\ref{lian}) w.r.t. $ F_1,F_2 \in K$:
$F_1 \circ L_2 + L_1 \circ F_2 =0$. The equation always has a
solution
\[
\left\{
  \begin{array}{ll}
    F_1= f_1(y)e^{(b-a)x}, \\
    F_2=-F_1,
  \end{array}
\right.
\]
where $f_1(y) \in K$ is a parameter function. Thus, any family can
be parameterized by only one function of one variable.

In fact, such a family always exists, and it is given explicitly in
the statement of the theorem. Moreover, one can prove
straightforwardly that such a family is unique for a given
operator $L$.
\end{proof}
\section{Third-Order Operators}

\begin{thm} \label{order3}
Let a third-order operator in $K[D_x,D_y]$ with the completely
factorable symbol has an irreducible family of factorizations. Then
it is almost ordinary.

Any such family depends by at most three (two) parameters if the
number of factors in factorizations is three (two). Each of these
parameters is a function of one variable.
\end{thm}
\begin{proof}
Consider a third-order operator $L$ in $K[D_x,D_y]$. For the symbol
$\Sym_L$ only the following three are possible: it has exactly
three, two, or no coprime factors. In the first case no family is
possible because of Theorem \ref{GS}.

Suppose exactly two factors of the symbol are coprime. Thus, in some
variables the symbol of $L$ is $X^2Y$. Consider factorization into
two factors. Then the following types of factorizations are
possible: $(X)(XY), (Y)(X^2), (XY)(X), (X^2)(Y)$. By Theorem
\ref{GS}, there is no family of factorizations of the types
$(Y)(X^2), (X^2)(Y)$. Because of the symmetry, it is enough to
consider just the case $(X)(XY)$. Indeed, if there exists a family
of the type $(XY)(X)$ for some operator $L$ of the general form
\eqref{general form}. Then the adjoint operator
\[
L^t (f) = \sum_{|J| \leq d} (-1)^{|J|} D^J (a_J f).
\]
has a family of the type $(X)(XY)$, and the number of parameters in
the family is the same.

 Thus, we consider a factorization of the
factorization type $(X)(XY)$:
\[
L=L_1 \circ L_2 = (D_x + r) \circ (D_{xy} + a D_x + b D_y +c),
\]
where $r,a,b,c \in K$ as the initial factorization for some family
of factorizations of the factorization type $(X)(XY)$. By means of
the gauge transformations, we make the coefficient $a$ equal zero in
this initial factorization (of course, the coefficient at $D_x$ in
the second factor of other factorizations of the family may be still
non-zero). To study possible families in this case, we consider the
linearized problem: the equation $F_1 \circ L_2 + L_1 \circ F_2 =0$
w.r.t. $F_1=r_1$, $F_2=a_{10} D_x + a_{01} D_y + a_{00}$, where
$r_1, a_{10}, a_{01}, a_{00} \in K$. The only non-trivial solution
is
\[
a_{10}=a_{00}=0, \ r_1=-a_{01}, \ a_{01}= f_1(y) \cdot Q,
\]
where $Q=e^{\int (b-r) dx}$ and $f_1(y) \in K$ is a parameter, while
\[
c=0
\]
is a necessary condition of the solution's existence. Therefore,
every family of the type $(X)(XY)$ is parameterized by one function
of one parameter (can be a constant function). Secondly, the initial
factorization has the form
\begin{equation} \label{(X)(XY)2}
L=( D_x + r ) \circ (D_x + b ) \circ D_y \ ,
\end{equation}
that is the operator $L$ itself has very special form.

Now, if we consider a factorization of the family in general form,
namely
\[
\widetilde{L}_1 \circ \widetilde{L}_2 = (D_x + \widetilde{r}) \circ
(D_{xy} + \widetilde{a} D_x + \widetilde{b} D_y +\widetilde{c}) \ ,
\]
where all the coefficients belong to $K$, and equates the
corresponding product to the expression \eqref{(X)(XY)2}, we obtain
\[
\widetilde{a}=\widetilde{c}=0,
\]
and so any factorization of the family has the form
\[
L=( D_x + \widetilde{r} ) \circ (D_x + \widetilde{b} ) \circ D_y \ .
\]
Therefore, only reducible families of factorizations into two
factors may exist in this case. Then, by Remark
\ref{number_of_factors}, there is no any irreducible family of
factorizations into any number of factors in this case.

Consider the case in which all the factors of the symbol $\Sym_L$
are the same (up to a multiplicative function from $K$). Then one
can find variables in which the symbol is $X^3$. Note that any
irreducible factorization of the factorization type $(X)(X)(X)$ is
an irreducible factorization of the types $(X)(X^2)$ and $(X^2)(X)$
also. Then because of the symmetry only one of two types $(X)(X^2)$
and $(X^2)(X)$ has to be considered. Therefore, it is sufficient to
consider the factorization type $(X)(X^2)$. Thus, consider an
initial factorization  of the form
\[
L=L_1 \circ L_2=(D_x + r) \circ (D_{xx} + a D_x + b D_y + c)
\]
where $r,a,b,c \in K.$ Under the gauge transformations we may assume
$a=0$ (while the coefficient at $D_x$ in the second factor of other
factorizations of the family may be still non-zero). Consider the
linearized problem \eqref{lian} for such the initial factorization:
the equation $F_1 \circ L_2 + L_1 \circ F_2 =0$ w.r.t. $F_1=r_1$,
$F_2=a_{10} D_x + a_{01} D_y + a_{00}$, where $r_1, a_{10}, a_{01},
a_{00} \in K$. The only non-trivial solution is $a_{01}=0$,
$r_1=-a_{10}$, $a_{00}=-ra_{10}-\partial_x(a_{10})$, provided both
\[
b=0
\]
and $c a_{10} + r^2 a_{10} + 2r \partial_x(a_{10})+a_{10}
\partial_x(r)+ \partial_{xx}(a_{10})=0$. The solution of the latter equation depends
on two arbitrary function in the variable $y$. Therefore, any family
of the type $(X)(XX)$ is parameterized by two functions of one
variable (can be constant functions), and such a family may exist
only for an \emph{almost ordinary} operator $L$. This implies that a
family of the factorization type $(X)(X)(X)$ may exist only for an
\emph{almost ordinary} operator $L$.

Any irreducible family of the type $(X)(X)(X)$ serves as an
irreducible family of the type $(X)(XX)$. Therefore, a family of the
type $(X)(X)(X)$ can have two parameters (functions in one
variables), that appear in the corresponding family of the type
$(X)(XX)$, and additional parameters, that can appear when we
consider two last factors separately. By the theorem \ref{order2},
there is at most one additional parameter (a function in one
variable). Thus, for the family of the type $(X)(X)(X)$ the maximal
number of parameters is three, and these parameters are functions in
one variable (may be constant functions). This agrees with
\cite{tsarev98}.
\end{proof}

The theorem implies that for an operator (with the completely
factorable symbol), that is not almost ordinary, only reducible
families may exist. Any such family is obtained by the
multiplication (on the left or on the right) of a second-order
family by some non-parametric first order operator. Note that this
second-order family should be almost ordinary, by Theorem
\ref{order2}.

\begin{ex} The family of the Landau operator \eqref{Landau_ex}
is reducible, which is obtained from a second-order family.
\end{ex}
%
%
\section{Fourth-Order Operators}

Here we start with an example of a fourth-order irreducible family for
an almost ordinary operator.
\begin{ex} \label{example_xx_xx}
The following is a fourth-order irreducible
family of factorizations:
\[
D_{xxxx}=\Big(D_{xx} + \frac{2}{x + 2 f_1(y)} +y \Big) \Big( D_{xx}
- \frac{2}{x + 2 f_1(y)} +y \Big),
\]
where $f_1(y) \in K$ is a parameter.
\end{ex}

Unlike the irreducible families of orders two and three, an
irreducible fourth-order family need not be almost ordinary.

\begin{ex} \label{example_xy_xy}
The following is a fourth-order irreducible
family of factorizations:
\[
D_{xxyy}=\Big(D_x + \frac{\alpha}{y + \alpha x + \beta} \Big) \Big(
D_y + \frac{1}{y + \alpha x + \beta} \Big) \Big(D_{xy} - \frac{1}{y
+ \alpha x + \beta} (D_x + \alpha D_y) \Big),
\]
where $\alpha, \beta \in K \backslash \{ 0 \}$. Note that the first
two factors commute.

Again, we actually have several examples here. Namely, for the same
operator $D_{xxyy}$, we have families of the types $(X)(XY^2)$,
$(XY)(XY)$ and $(X)(Y)(XY)$.
\end{ex}

\begin{thm} In $K[D_x,D_y]$, irreducible fourth order families of
factorizations with a completely factorable symbol exist only
if in some system of coordinates their factorization types are
$(XY)(XY)$, $(X)(XY^2)$, $(X)(Y)(XY)$ or $(X^2)(X^2)$,
and symmetric to them.
\end{thm}
\begin{proof} For the symbol $\Sym$ of the family,
 there are exactly four possibilities: to have exactly four ($\Sym=S_1S_2S_3S_4$),
 three ($\Sym=S_1^2S_2S_3$), two ($\Sym=S_1^2S_2^2$ and $\Sym=S_1S_2^3$), or no ($\Sym=S_1^4$) different factors.

 Consider factorizations into two factors first. Because of properties of factorizations
 of adjoint operators (see in more detail in the proof of Theorem \ref{order3}), it is enough to consider
 factorization type $(S_i)(S_j)$ instead of consideration of both $(S_i)(S_j)$, $(S_j)(S_i)$.
 Also, recall that by Theorem \ref{GS}, no family of type $(S_1)(S_2)$, where $S_1, S_2$ are coprime
 exists. Thus, it is enough to consider the following cases:
\begin{itemize}
  \item[I.] $\Sym=S_1^2S_2S_3$.
  \begin{itemize}
    \item[1)] $(S_1S_2)(S_1S_3)$.
    \item[2)] $(S_1)(S_1S_2S_3)$.
  \end{itemize}
  \item[II.] $\Sym=S_1^2S_2^2$.
  \begin{itemize}
    \item[1)] $(S_1S_2)(S_1S_2)$.
    \item[2)] $(S_1)(S_1S_2^2)$.
  \end{itemize}
  \item[III.] $\Sym=S_1S_2^3$.
  \begin{itemize}
    \item[1)] $(S_1S_2)(S_2^2)$.
    \item[3)] $(S_2)(S_1S_2^2)$.
  \end{itemize}
  \item[IV.] $\Sym=S_1^4$.
    \begin{itemize}
    \item[1)] $(S_1)(S_1^3)$.
    \item[3)] $(S_1^2)(S_1^2)$.
  \end{itemize}
\end{itemize}
\smallskip
\noindent
 $I$. Case $\Sym=S_1^2S_2S_3$. In this case, in
 appropriate variables, the symbol has the form $\Sym_L=X^2 Y(\alpha
 X + Y)$, where $\alpha \in K \backslash \{0 \}$.

  1) Case of the type $(XY)(X(\alpha X + Y))$. We prove that there
is no irreducible family of this type. Let $
L_1 \circ L_2 = (D_{xy} + a_1 D_x + b_1 D_y + c_1) \circ ( \alpha
D_{xx}+ D_{xy} + a_2 D_x + b_2D_y + c_2)$, where all coefficients are in $K$, be the initial factorization of
such a family. Then under the gauge transformation we may assume
$b_2=0$ (while the coefficient at $D_y$ in the second factor of any
other factorization of the family may be still non-zero). Consider
the linearized problem: the equation $F_1 \circ L_2 + L_1 \circ F_2
=0$ w.r.t. $F_1, F_2 \in K[D]$ and $\ord(F_1)=\ord(F_2)=1$. The only
non-trivial solution is parameterized by a function $f_1(y) \in K$
and exists only under two conditions on the coefficients of $L_1$
and $L_2$: $c_1= \partial_y(b_1)+ b_1 a_1$ and $c_2=\partial_{xx}(\alpha)- \partial_x(a_2)$.
Now we come back to the initial problem and look for a family of
factorizations in the general form: $L_1 \circ L_2 = (L_1 + S_1) \circ (L_2 + S_2)$,
where $S_1, S_2$ are arbitrary first-order operators in $K[D]$. This
gives us a system of equations in the coefficients of $S_1$ and
$S_2$. The system together with conditions
$c_1= \partial_y(b_1)+ b_1 a_1$ and $c_2=\partial_{xx}(\alpha)- \partial_x(a_2)$ has a
unique non-trivial solution. The corresponding (to this solution)
family of factorizations is complete, that is, both factors are
factorable themselves: $L_1 \circ L_2 = (D_y + a_1 ) \circ \Big(D_x +
b_1+\frac{Q}{W+f_1(y)} \Big)\! \circ \! \Big(D_x - \frac{Q}{W+f_1(y)}
\Big)\! \circ \! \Big( \alpha D_x + D_y + a_2-\partial_x(\alpha) \Big)$,
where $Q=e^{-\int b_1 dx}$ and $W=\int Q dx$ and $f_1(y)$ is the
only parameter function.
Now it is clear that the first and the
last factors do not depend on a parameter, and so any factorization
of any family of the type $(XY)(X(\alpha X + Y))$ is reducible.

 2) Case of the type $(S_1)(S_1S_2S_3)$. We prove now that there is no irreducible family
of this type. By a change of variables one can make  $S_1=X$, $S_2=Y$, $S_3=\alpha X+Y$, where $\alpha \in K$.
 Then consider a factorization of type $(S_1)(S_1S_2S_3)$,
\[
L_1 \circ L_2 = (D_x + c_1) \circ ( D_{xxy} + D_{xyy} + a D_{xx} + b
D_{xy} + c D_{yy} + d D_x + e D_y +f),
\]
where all the coefficients belong to $K$, as the initial
factorization of a family of factorizations. Under the gauge
transformations we may assume $c=0$ (note that the analogous
coefficients in the other factorizations of the family do not
necessary become zero). Proceeding as in the previous case, we also
get that there is only one non-trivial solution, which is
parameterized by a function $f_1(y) \in K$. Also we have two
conditions which provide the existence of such an equation:
$e_2 = b_2 c_2+\partial_x(b_2)-\alpha
\partial_x(c_2)-\alpha c_2^2-2 \partial_x(\alpha) c_2-\partial_{xx}(\alpha) $,
$f_2= d_2 c_2+\partial_x(d_2)-a_2
\partial_x(c_2)-a_2 c_2^2-2 \partial_x(a_2) c_2-\partial_{xx}(a_2)$.
Now, we use the obtained conditions for the initial problem, where a
family of factorizations is considered in general form. Thus, we get
that if such a family exists, then the second factor of the family
can be always factored into first and second-order operators, and
the second-order operator does not depend on the parameter:
$L_1 \circ L_2 = \Big(D_x + c_1 + m_{00} \Big) \circ \Big(D_x + c_2 - m_{00}\Big) \circ\\
  \Big( \alpha D_{x} + D_{y} + a_2 D_x + (b_2-\alpha
c_2-\partial_x(\alpha)) D_y + d_2-a_2 c_2- \partial_x(a_2) \Big)$,
 where only $m_{00} \in K$ may depend on a parameter. Thus any family
(if it exists) of the type $(X)(XY(\alpha X +Y))$ is reducible.

\smallskip \noindent
$II$. Case of symbol $\Sym=S_1^2S_2^2$.

  1) Case of the type $(S_1S_2)(S_1S_2)$. A family exists. See example \eqref{example_xy_xy}.

  2) Case of the type $(S_1)(S_1S_2^2)$. A family exists. See example \eqref{example_xy_xy}.

\smallskip \noindent
$III$. Case of symbol $\Sym=S_1S_2^3$.

  1) $(S_1S_2)(S_2^2)$.  In this case, in appropriate variables, the symbol has form $(XY)(Y^2)$. We prove that
  there is no irreducible family of this type. Indeed, consider a factorization of the considering type:
$L_1 \circ L_2 =(D_{xy} + a_1 D_x + b_1 D_y + c_1) \circ (D_{yy} +
a_2 D_x + b_2 D_y + c_2)$ with all coefficients in $K$, as the initial factorization of a
family of factorizations. By the gauge transformations we may assume
$c_2=0$. The linearized problem is the equation $F_1 \circ L_2 + L_1
\circ F_2 =0$ w.r.t. $F_1, F_2 \in K[D]$ and $\ord(F_1)=1, \
\ord(F_2)=1$. The equation has a non-trivial solution, which depends
on two parameter functions $f_1(x), f_2(x) \in K$, and the existence
is provided by the conditions $a_2=0$, $c_1=b_1 a_1 + \partial_x (a_1)$.
Thus, a family may exist only if the considered operator $L$ has the
form $L=(D_{xy} + a_1 D_x + b_1 D_y + b_1 a_1 + \partial_x (a_1)) \circ
(D_{y} + b_2) \circ  D_y$ for some $a_1,b_1,b_2 \in K$. Then, one may prove that in this case
any family of factorization has the form $ L=(D_{xy} + \dots ) \circ (D_{y}+ \dots) \circ D_y $,
meaning that is there is no irreducible fourth-order family of the type
$(XY)(Y^2)$.

  2) Case of type $(S_2)(S_1S_2^2)$. In this case, in
 appropriate variables, the symbol has form $(Y)(XY^2)$. We prove that there is no irreducible family
of this type. Indeed, consider a factorization of this type: $
L_1 \circ L_2 =(D_y + c_1) \circ (D_{xyy} +  a D_{xx} + b D_{xy} + c
D_{yy} + d D_x + e_1 D_y + f )$, with all coefficients in $K$, as the initial factorization of a
family of factorizations. By gauge transformations we may assume the
coefficient at $D_{yy}$ in this initial factorization is zero, that
is $c=0$. The linearized problem is the equation $F_1 \circ L_2 +
L_1 \circ F_2 =0$ w.r.t. $F_1, F_2 \in K[D]$ and $\ord(F_1)=0, \
\ord(F_2)=2$. This equation has a non-trivial solution, provided
$a=0$ and $d = e_1^{-2}(b f e_1 - b e_1 \partial_y(e_1) + e_1^2
\partial_y(b)+3f
\partial_y(e_1)- e_1 \partial_y(f)-2 (\partial_y(e_1))^2+ e_1
\partial_{yy}(e_1)-f^2)$. Then, when we consider the corresponding family of factorizations in
general form, we may apply these conditions, and easily get that
such a family cannot exist.

\smallskip
\noindent $IV$. Case of no different factors of the symbol. Then there
exist variables such that the symbol is $X^4$. Consider
factorizations into two factors. Then, by Theorem \ref{GS} and
because of the symmetry, it is enough to consider types of
factorizations $(X)(X^3)$ and $(X^2)(X^2)$.

1) Case of the type $(X)(X^3)$. Prove that there is no irreducible
family of the type  $(X)(X^3)$. Consider a factorization of the type
$(X)(X^3)$: $L_1 \circ L_2 =(D_{x} + c_1) \circ (D_{xxx} + a D_{xx}
+ b D_{xy} + c D_{yy} + d D_x + e D_y + f)$, where all coefficients are in $K$. Solving the linearized problem,
we get that a family in this case may be parameterized by only one function in one variable,
and such a family may exist provided two
conditions on the initial coefficients hold (one of them is just
$c=0$). Then, when we look for a family in general form, one may
prove that such families indeed can exist, but all such families are
reducible.

2) Case of the type $(X^2)(X^2)$. Here we have the Example
\ref{example_xx_xx} of a family of factorizations depended on one
functional parameter in one variable. In fact the maximal number of
parameters in this type of factorization is four \cite{tsarev98}.

\smallskip \noindent

Now consider factorizations into more than two factors. By Remark \ref{number_of_factors} and symmetry properties,
irreducible families into four and three factors cannot exist in cases $I$, $III$ and $IV$.
Consider parametric factorizations into three factors in case $II$,
$\Sym=S_1^2S_2^2$. It is enough to consider the following cases:
\begin{itemize}
  \item[a.] $(S_1)(S_2)(S_1S_2)$.
  \item[b.] $(S_1)(S_1 S_2)(S_2)$.
\end{itemize}

\noindent $a.$ A family exists. See example \eqref{example_xy_xy}.

\noindent $b.$ Prove that there is no irreducible family of type $(X)(XY)(Y)$ (in
appropriate variables, $S_1=X$ and $S_2=Y$). Consider an initial factorization of this
factorization type: $L=L_1 \circ L_2 \circ L_3 = (D_x+c_1)(D_{xy}+ a D_x + b D_y + c)(D_y + c_3)$,
where all the coefficients belong to $K$, and are known functions. Solving the linearized problem, i.e. \eqref{lian3},
we get $c=a_x + ab$. This means that $L_2=(D_x + b)\circ(D_y + a)$. As every factorization of a family
can be chosen as the initial, then the second factor is factorable for every factorization of the family.

Suppose there is another factorization of the same operator of the same factorization type, then
it has form $ L=M_1 \circ M_2 \circ M_3 \circ M_4= (D_x+m_1)(D_{x}+ m_{01})(D_y +  m_{01})(D_y + m_3)$,
where all the coefficients belong to $K$, and are some unknown functions. Equate the corresponding coefficients
of these two factorizations of $L$. Then, one can easily find expressions for $m_1, m_{01}, m_{01}, m_3$
in terms of $a,b,c_1,c_3$ and two parameter functions $F_1(x)$ and $F_2(y)$. However, if we compute the composition of
$M_1$ and $M_2$, all the parameters disappear, and, therefore,
any parametric family of factorization type $(S_1)(S_1 S_2)(S_2)$ is reducible.

Cases $a.$ and $b.$ also imply that irreducible families into four factors cannot exist
in case $II$.

\end{proof}
\section{Conclusion}

For second, third and fourth order LPDOs with completely factorable
symbols on the plane, we have completely investigated what
factorizations' types admit irreducible parametric factorizations.
For these factorization types, examples are given. Note the our
method is general and we cover the case of the ordinary operators as
a particular case. For operators of orders two and three, we
describe in addition the structure of their families of
factorizations. For the partial operators of order four, the
question remains open (for ordinary operators the possible number of
parameters in a family of factorizations has been investigated in
\cite{tsarev98}). For the case of partial differential operators we
would surmise that no more than two or one parameters (which could
be functions) are possible. Generalizations to LPDOs with arbitrary
symbols (without the complete factorization assumption), to high
order LPDOs, and to those in multiple-dimensional space are of
interest also.


\subsection*{Acknowledgment}
This work was supported by Austrian Science Foundation (FWF) under
the project SFB F013/F1304.
\end{document}